\documentclass[11pt,a4paper]{article}
\usepackage{geometry}
\usepackage{tikz}
\usetikzlibrary{shapes.geometric, decorations.pathreplacing}
\usepackage{amsmath}
\usepackage{amssymb}
\usepackage{amsthm}
\usepackage{subcaption}
\usepackage{mathtools}
\usepackage{amstext}
\usepackage{orcidlink}
\usepackage{amsfonts}
\usepackage{physics}
\usepackage{caption}
\usepackage{bbm}
\usepackage{mathrsfs}
\usepackage{upgreek}
\usepackage{multirow}
\usepackage{placeins}
\usepackage{hyperref}
\usepackage{bm}
\usepackage{enumitem}
\usepackage{xcolor}
\usepackage{wasysym}
\usepackage{pifont}
\usepackage{fancyhdr}
\usepackage{amsbsy}
\usepackage{stmaryrd}
\usepackage{cases}
\usepackage{wasysym}
\usepackage{appendix}
\usepackage{sidecap}
\usepackage{graphicx}
\usepackage{float}
\usepackage{dsfont}
\usepackage{indentfirst}
\usepackage{listings}
\usepackage{empheq,etoolbox}
\usepackage{nicefrac}

\theoremstyle{definition}
\newtheorem{definition}{Definition}[section]
\newtheorem{theorem}[definition]{Theorem}
\newtheorem{proposition}[definition]{Proposition}
\newtheorem{corollary}[definition]{Corollary}
\newtheorem{lemma}[definition]{Lemma}
\newtheorem{remark}[definition]{Remark}

\DeclareMathOperator{\mi}{i}
\DeclareMathOperator{\TTR}{TTR}
\DeclareMathOperator{\TR}{TR}

\newcommand{\B}{\mathcal{B}}
\newcommand{\N}{\mathbb{N}}
\newcommand{\R}{\mathbb{R}}
\newcommand{\C}{\mathbb{C}}
\newcommand{\eremk}{\hbox{}\hfill\rule{0.8ex}{0.8ex}}
\newcommand{\K}{\mathcal{K}}

\begin{document}

\title{Convolution quadrature based on a truncated trapezoidal rule}

\author{Matteo~Ferrari\thanks{Dipartimento di Matematica, Universit\`a di Pavia,  Pavia, Italy (\href{mailto:matteo.ferrari01@unipv.it}{m.ferrari@unipv.it})}  \hspace{0.05cm} \orcidlink{0000-0002-2577-1421}}

\date{}
\maketitle

\begin{abstract}
\noindent
We study the truncated trapezoidal rule, a family of A-stable second order multistep methods parametrized by an integer $J \ge 2$, a compromise between BDF2 and the trapezoidal rule. We obtain a closed-form expression for the coefficients that minimize the principal error constant under the A-stability constraint and we derive an explicit formula for the corresponding principal error constant. The latter decreases to the optimal Dahlquist value $1/12$ as $J$ increases, and is strictly smaller than the BDF2 constant for every $J \ge 2$. We apply the truncated trapezoidal rule within the convolution quadrature framework. Its analyticity in a neighbourhood of the closed unit disk yields milder regularity and perturbation requirements than those of the trapezoidal rule, while the error constant can be made arbitrarily close to the optimal one by increasing $J$. Numerical experiments show that convolution quadrature based on the truncated trapezoidal rule remains stable under symbol perturbations, where the trapezoidal rule fails, while achieving a smaller error constant than BDF2.
\end{abstract}

\section{Introduction}

\noindent 
Convolution Quadrature (CQ), introduced by Lubich in the seminal papers \cite{Lubich1988a, Lubich1988b}, is a technique for the numerical evaluation of one-sided convolutions
\begin{equation*}
    \mathcal{K}(\partial_t) g (t) = \int_0^t \kappa(t-\tau)\, g(\tau)\, \dd \tau, \quad 0 \le t \le T,
\end{equation*}
and for the numerical solution of convolution equations $\mathcal{K}(\partial_t) g = \phi$. Both the analysis and the implementation rely on the Laplace transform $\mathcal{K}$ of the kernel $\kappa$, rather than on $\kappa$ itself. CQ has become a standard tool in the time-domain treatment of integral equations arising from wave propagation, parabolic problems, viscoelasticity, fractional diffusion, and related areas; we refer to \cite{BanjaiSayas2022} for a comprehensive overview.

\smallskip
\noindent 
The key idea behind CQ is to discretize convolutions by means of an underlying ODE solver, whose generating function $\delta(\zeta)$ enters the formulation through a contour integral representation of the convolution weights. The choice of the multistep method \cite{Lubich1988a, Lubich1988b, Lubich1994, Lubich2004} or Runge-Kutta method \cite{LubichOstermann1993, BanjaiLubich2011, BanjaiLubichMelenk2011, BanjaiFerrari2022} determines the accuracy and stability of the resulting scheme.

\smallskip
\noindent 
In this work, we focus on CQ based on multistep methods. The stability of the discrete convolution requires the underlying multistep method to be A-stable, i.e., $\Re \delta(e^{\mi \theta}) \ge 0$ for all $\theta \in [0,2\pi]$, which is restricted by Dahlquist's second barrier \cite{dahlquist1963special} to convergence order at most two. Among A-stable second-order methods, common examples are BDF2 and the trapezoidal rule. The trapezoidal rule is non-dissipative ($\Re \delta(e^{\mi \theta})=0$) and achieves the optimal principal error constant $1/12$, but its generating function $\delta_{\TR}(\zeta) = 2(1-\zeta)/(1+\zeta)$ has a pole at $\zeta = -1$. This singularity leads to strict regularity requirements on the data for the application of CQ \cite{Banjai2010, ErusluSayas2020, BanjaiFerrari2026} and increases the sensitivity to perturbations of the symbol $\mathcal{K}$ \cite[\S 2.9]{BanjaiSayas2022}. BDF2, on the other hand, has an entire generating function and therefore requires milder regularity and perturbation requirements. However, it is dissipative with principal error constant $1/3$.

\smallskip
\noindent 
To overcome these drawbacks, \cite[\S 6]{Banjai2022} introduced the
\emph{Truncated Trapezoidal Rule of order $J$} (TTR$_J$) for integer
$J \ge 2$: a family of multistep methods which is a compromise between BDF2 and the trapezoidal rule, where $\delta_{\TR}$ is replaced by a polynomial of degree $J+1$ in $\zeta - 1$ with coefficients chosen to combine the small error constant of the trapezoidal rule with the analyticity of BDF2. Optimal coefficients in \cite{Banjai2022} were computed numerically for $J=4$. Our main result is an explicit closed-form formula for the optimal coefficients for any $J \ge 2$, under the criterion of minimizing the principal error constant of the method subject to A-stability. As a consequence, we obtain the explicit dependence of the principal error constant on $J$, which converges monotonically to the Dahlquist optimum $1/12$ as $J \to \infty$. We also show that the leading dissipation term vanishes for $J$ even. Numerical experiments show that TTR$_J$-based CQ is robust under symbol perturbations, unlike the trapezoidal rule, and has a smaller error constant than BDF2.

\smallskip
\noindent 
The paper is organized as follows. In Section~\ref{sec:2}, we recall the CQ framework for hyperbolic symbols, with emphasis on the convergence and stability properties of CQ based on BDF2 and the trapezoidal rule. Section~\ref{sec:3} introduces the truncated trapezoidal rule, states the main result (Theorem~\ref{th:3.1}), and derives some consequences. Section~\ref{sec:4} is devoted to the proof of Theorem~\ref{th:3.1}, and Section~\ref{sec:5} presents numerical experiments. Auxiliary identities involving Chebyshev polynomials of the second kind are collected in the appendix.

\section{Convolution quadrature based on multistep methods} \label{sec:2}

\noindent 
We briefly recall one-sided convolutions and their CQ discretization following \cite[Ch.~2]{BanjaiSayas2022}.

\subsection{One-sided convolution}

\noindent
We denote by $\mathcal{B}(X,Y)$ the space of bounded linear operators between Banach spaces $X$ and~$Y$. For $\mu \in \R$, the space $\mathcal{A}(\mu; \mathcal{B}(X,Y))$ of \textit{hyperbolic symbols} consists of analytic functions $\K :~\C_+ \to \mathcal{B}(X,Y)$, $\C_+ = \{ s \in \C : \Re s > 0\}$, satisfying
\begin{equation*}
	\| \K(s) \|_{\mathcal{B}(X,Y)} \le C_{\K}(\Re s)\, |s|^\mu, \quad s \in \C_+,
\end{equation*}
where $C_\K: \R^+ \to \R^+$ is a non-increasing function polynomially bounded close to $0^+$. Such symbols are the Laplace transforms of causal tempered distributions $\kappa$ with values in $\mathcal{B}(X,Y)$. We  denote by $\K(\partial_t) g$ the associated convolution acting on smooth causal data $g: \R \to X$
\begin{equation*}
    \K(\partial_t) g(t) = \int_0^t \kappa(t-\tau) g(\tau) \dd \tau, \qquad \text{where } \kappa = \mathcal{L}^{-1}\{\K\} \text{~is the inverse Laplace transform of $\K$}
\end{equation*}
when $\mu < -1$, and 
\begin{equation*}
    \K(\partial_t) g(t) = \partial_t^m \K_m(\partial_t) g(t), \qquad \text{where } \K_m(s) = s^{-m} \K(s)
\end{equation*}
when $\mu \geq -1$ (with $m$ any integer such that $\mu - m < -1$).

\smallskip
\noindent 
Given a linear $k$-step method with generating function $\delta(\zeta)$, A-stable and consistent of order $p$, the CQ approximation of $\K(\partial_t) g$ at $t_n = n\Delta t$ is the discrete convolution
\begin{equation} \label{eq:1}
	\K(\partial_t^{\Delta t})\, g(t_n) = \sum_{j=0}^n \omega_{n-j}(\K)\, g(t_j), \quad \text{with weights} \quad \K\!\left( \frac{\delta(\zeta)}{\Delta t} \right) = \sum_{j=0}^\infty \omega_j(\K)\, \zeta^j,
	\quad |\zeta| < 1.
\end{equation}
By the Cauchy integral formula, the weights admit the contour representation
\begin{equation} \label{eq:2}
    \omega_j(\K) = \frac{1}{2\pi \mi}\oint_{|\zeta| = \lambda} \K\!\left( \frac{\delta(\zeta)}{\Delta t} \right) \zeta^{-j-1}\, \dd \zeta, \qquad \lambda \in (0,1),
\end{equation}
which, after discretization by the trapezoidal rule on the contour, yields the efficient algorithm \cite{Lubich1988b} based on the FFT for the simultaneous evaluation of the weights.

\smallskip
\noindent
Two classical second-order A-stable methods are BDF2 and the trapezoidal rule,
\begin{equation} \label{eq:3}
    \delta_{\mathrm{BDF2}}(\zeta) = (1-\zeta) + \frac{1}{2}(1-\zeta)^2,
	\qquad \delta_{\TR}(\zeta) = 2\, \frac{1-\zeta}{1+\zeta},
\end{equation}
with principal error constants $1/3$ and $1/12$, respectively. They differ in an analyticity property: $\delta_{\mathrm{BDF2}}$ is entire, whereas $\delta_{\TR}$ has a pole at $\zeta = -1$. In particular, $\delta_{\mathrm{BDF2}}$ satisfies the property
\begin{equation} \label{eq:4}
    \delta(\zeta) \text{ is analytic in a neighbourhood of } \{|\zeta| \le 1\}.
\end{equation}
The performance of CQ depends on whether the generating function satisfies \eqref{eq:4}. The trapezoidal rule, despite its optimal error constant, is penalized in two complementary ways with respect to BDF2-like methods.

\paragraph{Convergence.} For $r \in \N$ we introduce the space
\begin{equation*}
    W_0^r(X) = \bigl\{ g \in C^{r-1}(\R; X) : g \text{ causal}, \; g^{(r)} \in L^1_{\mathrm{loc}}(\R; X), \; g^{(j)}(0) = 0 \text{ for } j = 0, \ldots, r-1 \bigr\}.\footnote{For data $g$ defined on $[0,T]$ with the required vanishing derivatives at $t = 0$, the extension to $W_0^r(X)$ is obtained by setting $g \equiv 0$ for $t < 0$ and by a Taylor polynomial of degree $r-1$ centered at $T$ for $t > T$.}
\end{equation*}

\noindent 
For $\K \in \mathcal{A}(\mu; \B(X,Y))$ with $\mu \ge 0$ and $g \in W_0^r(X)$, there exists a constant $C > 0$, depending on $t$, such that  (see \cite[Theorem 3.2]{Lubich1994}, \cite[Theorem 4]{Banjai2010}, and \cite[Theorem 2.1]{ErusluSayas2020})
\begin{equation} \label{eq:5}
    \| \K(\partial_t^{\Delta t}) g(t) - \K(\partial_t) g(t) \|_Y \le C\, \Delta t^p \int_0^t \| g^{(r)}(\tau) \|_X\, \dd \tau,
\end{equation}
provided that
\begin{equation} \label{eq:6}
	r > \mu + p + 2 \text{ under \eqref{eq:4}}, \qquad
	r > \max\{2\mu + 3, \mu + 4\} \text{ for the trapezoidal rule}.
\end{equation}

\paragraph{Stability under perturbations.} In practice the symbol $\K$ is rarely available exactly. Suppose we have computed a perturbed symbol $\K_\varepsilon \in \mathcal{A}(\mu_1,\mathcal{B}(X,X'))$ satisfying
\begin{equation} \label{eq:7}
    \| \K(s) - \K_\varepsilon(s) \|_{\mathcal{B}(X,X')} \le C_p(\Re s)\, |s|^{\mu'}\, \varepsilon, \qquad s \in \C_+
\end{equation}
where $\mu' \in \R$ and $C_p : (0,\infty) \to (0,\infty)$ is a non-increasing function polynomially bounded as $\Re s \to~0^+$. Such perturbations arise naturally from spatial discretizations, e.g., quadrature errors in a Galerkin scheme. 

\noindent
Let us assume that there exists $\zeta : \C_+ \to \C$, with $|\zeta(s)| =1$ for $s \in \C_+$,  and a function $D : (0,\infty) \to (0,\infty)$ non-decreasing and polynomially bounded as $\Re s \to 0^+$, such that
\begin{equation} \label{eq:8bis}
    \Re \langle \zeta(s)\varphi,\, \K(s)\varphi \rangle_{X \times X'} \;\ge\; D(\Re s)\, |s|^{-\mu_2}\, \|\varphi\|_X^2, \qquad \text{for all~} \varphi \in X,\ s \in \C_+.
\end{equation}
Then, the convolution equation $\K_\varepsilon(\partial_t^{\Delta t}) \varphi = g$ admits a unique solution and the perturbation error is $\mathcal{O}(\varepsilon)$, provided $\varepsilon$ is small enough. The threshold depends on whether \eqref{eq:4} holds:
\begin{equation*}
	\varepsilon \lesssim \Delta t^{\mu'+\mu_2} \text{ under \eqref{eq:4}}, \qquad
	\varepsilon \lesssim \Delta t^{2(\mu'+\mu_2)} \text{ for the trapezoidal rule}.
\end{equation*}
The gap reflects the different growth of the discrete symbol $s^{\Delta t}(\zeta) = \delta(\zeta)/\Delta t$, which is $\mathcal{O}(\Delta t^{-1})$ under \eqref{eq:4}, but is $\mathcal{O}(\Delta t^{-2})$ for the trapezoidal rule, due to the pole of $\delta_{\TR}$ at $\zeta = -1$. We refer to \cite[Theorem 2.38]{BanjaiSayas2022} for the precise statement.

\noindent
The contrast in Table~\ref{tab:1} compares BDF2 and the trapezoidal rule. 
The trapezoidal rule achieves the optimal error constant $1/12$ but, due to the pole of $\delta_{\TR}$ at $\zeta = -1$, requires data regularity scaling as $2\mu$ and tolerates perturbations only of order $\Delta t^{2(\mu'+\mu_2)}$. BDF2 satisfies \eqref{eq:4}, hence relaxes both requirements, but has a larger error constant. The truncated trapezoidal rule, introduced in the next section, combines \eqref{eq:4} with an error constant arbitrarily close to $1/12$.

\begin{table}[ht]
\centering
\renewcommand{\arraystretch}{1.4}
\begin{tabular}{l|c|c}
&  BDF2 & trapezoidal rule \\
\hline
    principal error constant
	& $1/3$
	& $1/12$ \\
    regularity for \eqref{eq:5}
	& $r > \mu + 4$
	& $r > \max\{2\mu + 3,\, \mu + 4\}$ \\
    perturbation threshold
	& $\varepsilon \lesssim \Delta t^{\mu'+\mu_2}$
	& $\varepsilon \lesssim \Delta t^{2(\mu'+\mu_2)}$ \\
\end{tabular}
\caption{CQ based on BDF2 and the trapezoidal rule: a comparison.}
\label{tab:1}
\end{table}

\section{\texorpdfstring{Truncated trapezoidal rule of order $J$}{}} \label{sec:3}

\noindent 
For any integer $J \ge 2$, let us consider the generating function $\delta_{\TTR_J} : \C \to \C$ of a linear multistep method, defined as
\begin{equation} \label{eq:8}
    \delta_{\TTR_J}(\zeta) = (1-\zeta) + \frac{1}{2}(1-\zeta)^2 + \sum_{\ell=2}^J 2^{-\ell} c_\ell^J (1-\zeta)^{\ell+1},
\end{equation}
with coefficients $\{c_\ell^J \}_{\ell=2}^J \subset \R$. We call the associated multistep method the \textit{truncated trapezoidal rule of order $J$}. It was introduced in \cite[\S 6]{Banjai2022} (see also \cite[\S 4.13]{BanjaiSayas2022}) as a compromise between the BDF2 method and the trapezoidal rule.

\noindent 
Choosing any integer $J \ge 2$ and $c_\ell^J=0$ for all $\ell=2,\ldots,J$, yields the BDF2 method (recall \eqref{eq:3}), whereas setting $c_\ell^J=1$ for all $\ell=2,\ldots,J$ and letting $J \to +\infty$ recovers the trapezoidal rule. Indeed, we have
\begin{equation*}
    (1-\zeta) + \frac{1}{2}(1-\zeta)^2 + \sum_{\ell=2}^\infty 2^{-\ell} (1-\zeta)^{\ell+1} = (1-\zeta)\sum_{\ell=0}^\infty 2^{-\ell} (1-\zeta)^{\ell} = 2 \frac{1-\zeta}{1+\zeta} = \delta_{\TR}(\zeta).
\end{equation*}
We consider the restriction of $\delta_{\TTR_J}$ to the boundary of the unit circle. We have, for $\theta \in [0,2\pi]$,
\begin{equation} \label{eq:9}
    \delta_{\TTR_J}(e^{\mi \theta}) = \sum_{k=1}^{\infty} \mi^k C_k^J \theta^k = -\mi \theta + \mi \left( -\frac{1}{3} + \frac{c_2^J}{4} \right)\theta^3 + \left(\frac{1}{4} - \frac{3 c_2^J}{8} + \frac{c_3^J}{8} \right) \theta^4 + \sum_{k=5}^\infty \mi^k C_k^J \theta^k, 
\end{equation}
for some coefficients $C_k^J \in \R$. In particular, we have
\begin{equation} \label{eq:10}
    C_1^J = -1, \quad\quad C_2^J = 0, \quad\quad C_3^J = \frac{1}{3} - \frac{c_2^J}{4}, \quad\quad C_4^J = \frac{1}{4} - \frac{3c_2^J}{8} + \frac{c_3^J}{8}.
\end{equation}
The coefficient $C_3^J$ is the \textit{principal error constant} of the method. Indeed, since $C_2^J = 0$, the truncated trapezoidal rule is of order $2$, and smaller $|C_3^J|$ corresponds to a smaller leading-order truncation error. As reference values, the BDF2 method has $c_2^J=0$ and yields $C_3^{\mathrm{BDF2}} = 1/3$, whereas the trapezoidal rule has $c_2^J=1$ and yields $C_3^{\TR} = 1/12$. For application to convolution quadrature, we also need A-stability. We want to ensure that
\begin{itemize}
\item the principal error constant $|C_3^J|$ is small,
\item the corresponding method is \textit{A-stable}, meaning that 
\begin{equation} \label{eq:11}
    \text{Re}\bigl(\delta_{\TTR_J}(e^{\mi \theta})\bigr) \ge 0 \quad \text{for all } \theta \in [0,2\pi].
\end{equation}
\end{itemize}
It is known that the trapezoidal rule is the second-order A-stable method with the optimal principal error constant  \cite{dahlquist1963special}.

\smallskip
\noindent 
Our goal is, given an integer $J \ge 2$, to find optimal coefficients $\{c_\ell^J\}_{\ell=2}^J$. We define the coefficients $\{c_\ell^J\}_{\ell=2}^J$ to be \textit{optimal} if the A-stability condition \eqref{eq:11} is satisfied and $c_2^J$ is maximized. This latter corresponds to small $|C_3^J|$. The main result of this paper is an explicit formula for the optimal choice of coefficients $\{c_\ell^J\}^J_{\ell=2}$ for any given integer $J \ge 2$.

\begin{theorem} \label{th:3.1}
For a given integer $J \ge 2$, the optimal parameters $\{c_\ell^J\}_{\ell=2}^J$ of the TTR$_J$ are given by 
\begin{equation} \label{eq:12}
\begin{aligned}
    c_2^J & = \frac{2 \cos x_J}{1+\cos x_J},
    \\ c_\ell^J & = (-1)^{J+1-\ell}\, \frac{2^{\ell-1}(1-\cos x_J)}{J+1} \sum_{k=1}^{J-2} (-1)^k \binom{J-k-2}{\ell-3} U'_k(\cos x_J), \quad\quad \ell \ge 3.
\end{aligned}
\end{equation}
where $x_J = \displaystyle{\frac{\pi}{J+1}}$ and  $U_k$ denotes the $k$-th Chebyshev polynomial of the second kind.
\end{theorem}
\noindent 
In Section \ref{sec:4}, we provide a proof for Theorem \ref{th:3.1}. Here, we report some consequences.

\noindent 
We recall that for $x \in [0,2\pi]$ and $k \in \N$, we have (see, e.g., \cite[Def. 1.2]{Mason2002})
\begin{equation} \label{eq:13}
    U_k(\cos x) = \frac{\sin ((k+1) x)}{\sin x},
\end{equation}
and differentiating with respect to $x$, we get
\begin{equation} \label{eq:14}
    U'_k(\cos x) = - \frac{1}{\sin x} \left.\frac{\dd U_k(y)}{\dd y}\right|_{y=\cos x} = \frac{\sin((k+1)x)\cos x - (k+1)\cos((k+1)x)\sin x}{\sin^3 x}.
\end{equation}
\begin{remark}
In Table \ref{tab:2}, we report the values obtained with \eqref{eq:12} for small values of $J$. 

\begin{table}[H]
\centering
\renewcommand{\arraystretch}{2.5} 
\begin{tabular}{|c|c|c|c|c|c|}
\hline
$J$ & $c_2^J$ & $c_3^J$ & $c_4^J$ & $c_5^J$ & $c_6^J$ \\ \hline
$2$ & $\displaystyle{\frac{2}{3}}$ & & & & \\ \hline
$3$ & $2\sqrt{2}-2$ & $2-\sqrt{2}$ & & & \\ \hline
$4$ & $\displaystyle{\frac{2\sqrt{5}}{5}}$ & $\displaystyle{\frac{6\sqrt{5}-10}{5}}$ & $\displaystyle{\frac{12-4\sqrt{5}}{5}}$ & & \\ \hline
$5$ & $4\sqrt{3}-6$ & $\displaystyle{\frac{44-24\sqrt{3}}{3}}$ & $\displaystyle{\frac{24\sqrt{3}-40}{3}}$ & $\displaystyle{\frac{16-8\sqrt{3}}{3}}$ & \\ \hline
$6$ & $0.9479$ & $0.8437$ & $0.7998$ & $0.2734$ & $0.9054$ \\ \hline
\end{tabular}
\caption{Optimal parameters $\{c_\ell^J\}_{\ell=2}^J$.}
\label{tab:2}
\end{table}

\noindent 
The values for $J=4$ correspond to
\begin{equation*}
    c_2^4 \approx 0.894427190999916, \quad c_3^4 \approx  0.683281572999748, \quad c_4^4 \approx 0.611145618000168
\end{equation*}
obtained numerically in \cite[\S 6]{Banjai2022}.
\eremk
\end{remark}
\begin{corollary}
For a given integer $J \ge 2$, the optimal coefficient $c_2^J$ is given by
\begin{equation*}
    c_2^J = 1 - \tan^2\left(\frac{x_J}{2} \right),
\end{equation*}
where $x_J = \displaystyle{\frac{\pi}{J+1}}$. The corresponding principal error constant $|C_3^J|$ is 
\begin{equation} \label{eq:15}
    |C_3^J| = \frac{1}{12} + \frac{1}{4}\tan^2\left(\frac{x_J}{2}\right).
\end{equation}
\end{corollary}
\begin{proof}
Using the identity $(1-\cos x)/(1+\cos x) = \tan^2 (x/2)$ and \eqref{eq:12}, we readily obtain
\begin{equation*}
    c_2^J = 1 - \frac{1-\cos x_J}{1+\cos x_J} = 1 - \tan^2\left(\frac{x_J}{2}\right).
\end{equation*}
Finally, we substitute $c_2^J$ into $|C_3^J|$ in \eqref{eq:10} to deduce \eqref{eq:15}.
\end{proof}
\noindent 
This result is consistent with Dahlquist's classical result \cite{dahlquist1963special}. Indeed, from \eqref{eq:15}, for all $J \ge 2$, we have $|C_3^J| > 1/12$, and in the limit as $J \to +\infty$, we recover $|C_3^J| \to 1/12$. Moreover, this alternative expression of $c_2^J$ shows that $c_2^J$ increases as $J$ increases and $c_2^J \to 1$. 

\smallskip
\noindent 
In the next corollary, we show that even though the methods associated with $\delta_{\TTR_J}$ are not conservative, i.e., $\Re(\delta_{\TTR_J}(e^{\mi \theta})) > 0$, the dissipation error is of order at least $6$ when $J$ is even.

\begin{corollary}
For a given integer $J \ge 2$, the optimal $c_3^J$ is given by
\begin{equation} \label{eq:16}
    c_3^J = 3 c_2^J - 2 + \frac{1-(-1)^J}{J+1}(1-c_2^J)(2-c_2^J).
\end{equation}
In particular, $C_4^J = 0$ if $J$ is even.
\end{corollary}
\begin{proof}
We evaluate the optimal $c_3^J$ with \eqref{eq:12} by substituting $\ell=3$. We obtain
\begin{equation} \label{eq:17}
    c_3^J = (-1)^J \frac{4(1-\cos x_J)}{J+1} \sum_{k=1}^{J-2} (-1)^k U'_k(\cos x_J).
\end{equation}
In Lemma \ref{lem:A.5} below, we show that
\begin{equation*}
    \sum_{k=1}^{J-2} (-1)^k U'_k(\cos x_J) = \frac{(-1)^J (J+1)(2\cos x_J - 1)}{2(1-\cos x_J)(1+\cos x_J)} + \frac{(-1)^J - 1}{2(1+\cos x_J)^2}.
\end{equation*}
Thus, inserting in \eqref{eq:17}, we compute
\begin{equation*}
\begin{aligned}
    c_3^J = \frac{2(2\cos x_J-1)}{1+\cos x_J} + \frac{1-(-1)^J}{(J+1)} \frac{2(1-\cos x_J)}{(1+\cos x_J)^2}.
\end{aligned}
\end{equation*}
The first term, recalling \eqref{eq:12}, reads
\begin{equation*}
    \frac{2(2\cos x_J - 1)}{1+\cos x_J} = 3\left(\frac{2\cos x_J}{1+\cos x_J}\right) - 2 = 3c_2^J - 2.
\end{equation*}
For the second term, we observe that
\begin{equation*}
    (1-c_2^J)(2-c_2^J) = \frac{2(1-\cos x_J)}{(1+\cos x_J)^2}.
\end{equation*}
Substituting these identities, we conclude
\begin{equation*}
    c_3^J = 3c_2^J - 2 + \frac{1-(-1)^J}{J+1}(1-c_2^J)(2-c_2^J).
\end{equation*}
Finally, recalling \eqref{eq:10}, we deduce $C_4^J=0$ if $J$ is even.
\end{proof}
\begin{remark}
From \eqref{eq:12} and \eqref{eq:16}, as $J \to +\infty$, we have
$c_2^J \to 1, \quad c_3^J \to 1$. Moreover, by \eqref{eq:15}, the principal error constants $|C_3^J|$ converge to $1/12$, the optimal value among A-stable second-order multistep methods, achieved only by the trapezoidal rule \cite{dahlquist1963special}. In this sense, the family $\delta_{\TTR_J}$ approaches $\delta_{\TR}$ as $J \to \infty$.

\smallskip
\noindent 
Nevertheless, the convergence is non-uniform in the coefficients. Indeed, from \eqref{eq:12} with $\ell = J$, we have
\begin{equation*}
    c_J^J = \frac{2^{J-1}(1-\cos x_J)}{J+1} U'_1(\cos x_J) = \frac{2^J(1-\cos x_J)}{J+1} \to +\infty 
\end{equation*}
as $J \to +\infty$.
\eremk
\end{remark}
\noindent 
Now we discuss the stability region area of the methods associated with optimal parameters.

\smallskip
\noindent 
Since the methods are A-stable for all $J \ge 2$, the stability region contains the entire left half-plane. We are interested in the area $\mathcal{A}_J$ of the instability region located in the right half-plane. 

\smallskip
\noindent 
The instability region corresponds to the image of the unit disk under the mapping defined by the generating function $\delta_{\TTR_J}$. We report in Figure \ref{fig:1} the instability regions for $J=2,3,4,5$.

\begin{figure}[H]
    \centering
    \includegraphics[width=0.625\linewidth]{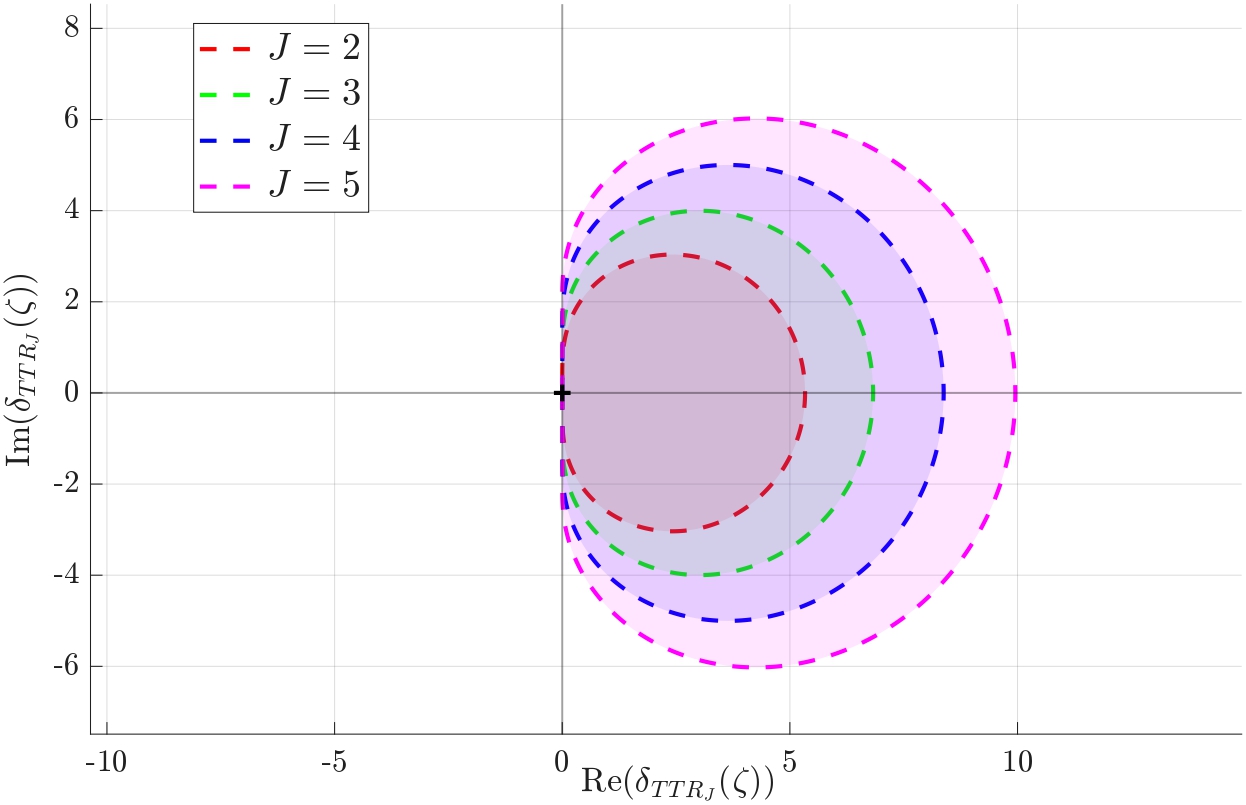}
    \caption{Instability regions of $\delta_{\TTR_J}$ for $J=2, 3, 4$, and $5$.}
    \label{fig:1}
\end{figure}

\noindent 
To compute the areas, we recall the classical \textit{Area Theorem} (see, e.g., \cite[Pag. 250]{Rudin1987}). If a polynomial $P(\zeta) = \sum_{m=0}^N p_m \zeta^m$ maps the unit disk onto a domain and it is one-to-one, the area of that mapped region is given by
\begin{equation} \label{eq:18}
    \mathcal{A} = \pi \sum_{m=1}^N m |p_m|^2.
\end{equation}
To identify the coefficients $\{p_m^J\}_{m=0}^{J+1}$ of $\delta_{\TTR_J}$ and the relation with $\{c_\ell^J\}_{\ell=2}^J$, we rewrite the generating function as a polynomial in $\zeta$. Expanding the terms $(1-\zeta)^k$ using the binomial theorem and collecting the powers of $\zeta$, we obtain $\delta_{\TTR_J}(\zeta) = \sum_{m=0}^{J+1} p_m^J \zeta^m$ with coefficients
\begin{equation*}
    p_m^J = (-1)^m \left( \binom{1}{m} + \frac{1}{2} \binom{2}{m} + \sum_{\ell=2}^J 2^{-\ell} c_\ell^J \binom{\ell+1}{m} \right), \quad \text{for } m = 0, \dots, J+1,
\end{equation*}
with the convention that $\displaystyle{\binom{n}{k} = 0}$ if $k > n$. Applying \eqref{eq:18}, we obtain 
\begin{equation*}
    \mathcal{A}_J = \pi \sum_{m=1}^{J+1} m \left( \binom{1}{m} + \frac{1}{2} \binom{2}{m} + \sum_{\ell=2}^J 2^{-\ell} c_\ell^J \binom{\ell+1}{m} \right)^2.
\end{equation*}
This formula allows for explicit evaluation of the area for any given $J \ge 2$ using the optimal coefficients computed in Theorem \ref{th:3.1}. We report in Table~\ref{tab:3} the values for $J = 2, 3, 4, 5$.

\begin{table}[H]
\centering
\begin{tabular}{|c|c|c|c|c|}
\hline
$J$ & $2$ & $3$ & $4$ & $5$ \\ \hline
Area $\mathcal{A}_J$ \rule[-20pt]{0pt}{45pt} & $\displaystyle\frac{25\pi}{3}$ & $\displaystyle\frac{\pi(34\sqrt{2} + 65)}{8}$ & $\displaystyle\frac{19\pi(2\sqrt{5} + 7)}{10}$ & $\displaystyle\frac{13\pi(8\sqrt{3} + 15)}{12}$ \\ \hline
\end{tabular}
\caption{Area of the instability region of $\delta_{\TTR_J}$ for $J=2,3,4,5$.}
\label{tab:3}
\end{table}

\noindent 
Finally, we explicitly compute the value of the $\delta_{\TTR_J}(-1)$. We recall that this point corresponds to a simple pole for the classical trapezoidal rule, whereas the truncated generating function remains finite for any finite $J$.
\begin{proposition} \label{prop:3.6}
For every integer $J \ge 2$, with optimal coefficients $\{c_\ell^J\}_{\ell=2}^J$ given by Theorem~\ref{th:3.1}, it holds true
\begin{equation*}
    \delta_{\TTR_J}(-1) = \frac{8}{(J+1)(1-\cos x_J)},
\end{equation*}
with $\displaystyle{x_J = \frac{\pi}{J+1}}$.
\end{proposition}
\begin{proof}
Evaluating $\delta_{\TTR_J}$ as in \eqref{eq:8} at $\zeta = -1$ gives
\begin{equation} \label{eq:19}
    \delta_{\TTR_J}(-1) = 4 + 2 c_2^J + 2 \sum_{\ell=3}^J c_\ell^J.
\end{equation}
Inserting \eqref{eq:12} into the sum and interchanging the order of summation,
\begin{equation*}
    \sum_{\ell=3}^J c_\ell^J = \frac{1-\cos x_J}{J+1} \sum_{k=1}^{J-2} (-1)^k U'_k(\cos x_J) \sum_{\ell=3}^J (-1)^{J+1-\ell}\, 2^{\ell-1} \binom{J-k-2}{\ell-3}.
\end{equation*}
The inner sum evaluates to
\begin{equation*}
    \sum_{\ell=3}^J (-1)^{J+1-\ell}\, 2^{\ell-1} \binom{J-k-2}{\ell-3} = 4(-1)^J \sum_{j=0}^{J-k-2} \binom{J-k-2}{j} (-2)^j = 4(-1)^k,
\end{equation*}
hence we deduce
\begin{equation*}
    \sum_{\ell=3}^J c_\ell^J = \frac{4(1-\cos x_J)}{J+1} \sum_{k=1}^{J-2} U'_k(\cos x_J).
\end{equation*}
By Lemma~\ref{lem:A.4}, we have
\begin{equation*}
\begin{aligned}
    \sum_{\ell=3}^J c_\ell^J & = \frac{4(1-\cos x_J)}{J+1} \frac{2(J+1)\cos^2 x_J - (J-1)(1+\cos x_J)}{2(1-\cos x_J)^2(1+\cos x_J)} 
    \\ & = \frac{2\bigl(2(J+1)\cos^2 x_J - (J-1)(1+\cos x_J)\bigr)}{(J+1)(1-\cos x_J)(1+\cos x_J)}
\end{aligned}
\end{equation*}
and therefore with \eqref{eq:12}, we obtain
\begin{equation*}
    2 c_2^J + 2 \sum_{\ell=3}^J c_\ell^J = \frac{4\bigl((J+1)\cos x_J - (J-1)\bigr)}{(J+1)(1-\cos x_J)}.
\end{equation*}
Substituting into \eqref{eq:19}, we get
\begin{equation*}
\begin{aligned}
    \delta_{\TTR_J}(-1) & = 4 + \frac{4\bigl((J+1)\cos x_J - (J-1)\bigr)}{(J+1)(1-\cos x_J)} 
    \\ & = \frac{8}{(J+1)(1-\cos x_J)}.
\end{aligned}
\end{equation*}
\end{proof}
\begin{remark} \label{rem:3.7}
For every $J \ge 2$, the value $\delta_{\TTR_J}(-1)$ grows linearly in $J$.  Indeed, from $1 - \cos x \ge 9 x^2 / 20$ valid for 
$|x| \le \pi/3$, Proposition~\ref{prop:3.6} yields
\begin{equation*}
    \delta_{\TTR_J}(-1) \le \frac{160(J+1)}{9\pi^2},
\end{equation*}
with sharp asymptotic constant $\delta_{\TTR_J}(-1) \sim 16(J+1)/\pi^2$ as $J \to \infty$. As $J \to \infty$, we deduce $\delta_{\TTR_J}(-1) \to \infty$, consistently with the pole of $\delta_{\TR}$ at $\zeta = -1$.
\eremk
\end{remark}
\noindent
We close this section with an elementary auxiliary result. Whenever the coefficients $\{c_\ell^J\}_{\ell=2}^J$ are non-negative, which is true for optimal ones by Table~\ref{tab:2} up to $J=6$, the supremum of $|\delta_{\TTR_J}(\zeta)|$ on $\{|\zeta|=1\}$ is attained at $\zeta = -1$, and its value is given explicitly by Proposition~\ref{prop:3.6}.
\begin{lemma} \label{lem:3.8}
Let $J \ge 2$ be an integer and assume $c_\ell^J \ge 0$ for all $\ell \in \{2, \ldots, J\}$. Then, we have
\begin{equation*}
    \sup_{|\zeta| \le 1} |\delta_{\TTR_J}(\zeta)| = \delta_{\TTR_J}(-1).
\end{equation*}
\end{lemma}
\begin{proof}
For every $\zeta \in \C$ with $|\zeta| \le 1$, the triangle inequality yields $|1-\zeta| \le 1 + |\zeta| \le 2$, with equality if and only if $\zeta = -1$. Define $\Phi_J : [0,2] \to \R^+$ as
\begin{equation*}
    \Phi_J(r) = r + \frac{1}{2} r^2 + \sum_{\ell=2}^J 2^{-\ell} c_\ell^J\, r^{\ell+1}.
\end{equation*}
Since all coefficients of $\Phi_J$ are non-negative, $\Phi_J$ is monotone non-decreasing. From the definition of $\delta_{\TTR_J}$ in \eqref{eq:8} and the triangle inequality, we deduce
\begin{equation*}
    |\delta_{\TTR_J}(\zeta)| \le |1-\zeta| + \frac{1}{2}|1-\zeta|^2 + \sum_{\ell=2}^J 2^{-\ell} c_\ell^J\, |1-\zeta|^{\ell+1} = \Phi_J(|1-\zeta|) \le \Phi_J(2) =\delta_{\TTR_J}(-1),
\end{equation*}
where we used the monotonicity of $\Phi_J$.
\end{proof}

\section{Proof of Theorem \ref{th:3.1}} \label{sec:4}

\noindent 
Theorem \ref{th:3.1} is a consequence of the following classical result \cite{fejer1915uber} (see also \cite[Satz I]{egervary-szasz-1928}).
\begin{theorem} \label{th:4.1}
Let $n \in \N$, $\{a_k\}_{k=1}^n, \{b_k\}_{k=1}^n \subset \R$, and let
\begin{equation} \label{eq:20}
    T(\theta) = 1 + \sum_{k=1}^n (a_k \cos(k\theta) + b_k \sin(k\theta))
\end{equation}
be non-negative in $[0,2\pi]$. Then, it holds true
\begin{equation} \label{eq:21}
    \sqrt{a_1^2 + b_1^2} \le 2\cos\left(\frac{\pi}{n + 2}\right).
\end{equation}
Moreover, equality in \eqref{eq:21} is attained if and only if $T(\theta)$ is of the form
\begin{equation} \label{eq:22}
    T(\theta) = 1 + \frac{2}{n+2} \sum_{k=1}^n \Bigg( (n - k + 1)\cos \Big(k\frac{\pi}{n+2}\Big) + \frac{\sin\big((k+1)\frac{\pi}{n+2}\big)}{\sin\big(\frac{\pi}{n+2}\big)} \Bigg) \cos(k(\theta - \psi))
\end{equation}
where $\psi$ is an arbitrary constant.
\end{theorem}
\noindent 
First, we connect $\Re (\delta_{\TTR_J}(e^{\mi \theta}))$ in \eqref{eq:9} to a trigonometric polynomial of the form \eqref{eq:20}. 
\begin{lemma}
For any integer $J \ge 2$, we have 
\begin{equation} \label{eq:23}
    \Re \bigl(\delta_{\TTR_J}(e^{\mi\theta})\bigr) = \sin^4 \left( \frac{\theta}{2} \right) T^J(\theta), \quad\quad \theta \in [0,2\pi]
\end{equation}
with the trigonometric polynomial
\begin{equation} \label{eq:24}
    T^J(\theta) = 4 - 2c_2^J -4c_2^J \cos \theta +  \sum_{k=2}^{J-1} A_k^J \cos(k\theta)
\end{equation}
and coefficients
\begin{equation} \label{eq:25}
    A_k^J = (-1)^k \sum_{\ell=k+1}^J c_\ell^J 2^{4-\ell} \binom{\ell-3}{k-2} \quad\quad  \text{for~} k \ge 2.
\end{equation}
\end{lemma}
\begin{proof}
By direct calculation, we have
\begin{equation*}
    \Re\left( (1-e^{\mi \theta}) + \frac{1}{2}(1-e^{\mi \theta})^2 \right)  = 4 \sin^4 \left(\frac{\theta}{2} \right).
\end{equation*}
Next, we evaluate the third term in $\Re (\delta_{\TTR_J}(e^{\mi \theta}))$
\begin{equation*}
    \Re\left( 2^{-2} c_2^J (1-e^{\mi \theta})^3 \right) = \sin^4 \left(\frac{\theta}{2} \right) (-2 c_2^J - 4 c_2^J \cos \theta).
\end{equation*}
Combining these initial terms, we obtain the first two coefficients of $T^J(\theta)$ in \eqref{eq:24}. It remains to show
\begin{equation*}
    \sum_{\ell=3}^J 2^{-\ell} c_\ell^J \Re \bigl((1-e^{\mi \theta})^{\ell+1}\big) = \sin^4 \left( \frac{\theta}{2}  \right) \sum_{k=2}^{J-1}  A_k^J \cos (k \theta)
\end{equation*}
with $A_k^J$ as in \eqref{eq:25}. We have  
\begin{equation*}
\begin{aligned}
    \Re \bigl((1-e^{\mi \theta})^{\ell+1}\bigr) & = \frac{(1-e^{\mi \theta})^{\ell+1} + (1-e^{-\mi \theta})^{\ell+1}}{2}
    \\ & = \frac{(1-e^{\mi \theta})^{\ell+1} \bigl(1 + (-1)^{\ell+1} e^{-\mi (\ell+1)\theta} \bigr)}{2}.
\end{aligned}
\end{equation*}
We can also write
\begin{equation*}
    \sin^4\left(\frac{\theta}{2}\right) = \left( \frac{e^{\mi \frac{\theta}{2}} - e^{-\mi \frac{\theta}{2}}}{2\mi} \right)^4 = \frac{e^{-\mi 2\theta}(1-e^{\mi \theta})^4}{16},
\end{equation*}
from which
\begin{equation*}
    \frac{\Re \bigl((1-e^{\mi \theta})^{\ell+1}\bigr)}{\sin^4 \left( \frac{\theta}{2} \right)} = 8  (1-e^{\mi \theta})^{\ell-3} \bigl( e^{\mi 2\theta} + (-1)^{\ell+1} e^{\mi (1-\ell)\theta} \bigr).
\end{equation*}
Next, we expand $(1-e^{\mi \theta})^{\ell-3}$,
distributing into two separate sums, and we compute
\begin{equation*}
\begin{aligned}
    \frac{\Re \bigl((1-e^{\mi \theta})^{\ell+1}\bigr)}{\sin^4 \left( \frac{\theta}{2} \right)} & = 8  \left( \sum_{m=0}^{\ell-3} \binom{\ell-3}{m} (-1)^m e^{\mi (m+2)\theta} + \sum_{m=0}^{\ell-3} \binom{\ell-3}{m} (-1)^{m+\ell+1} e^{\mi (m+1-\ell)\theta} \right)
    \\ & = 8  \left( \sum_{k=2}^{\ell-1} \binom{\ell-3}{k-2} (-1)^k e^{\mi k \theta} + \sum_{k=2}^{\ell-1} \binom{\ell-3}{\ell-k-1} (-1)^k e^{-\mi k \theta} \right)
    \\ & = 8 \sum_{k=2}^{\ell-1} \binom{\ell-3}{k-2} (-1)^k ( e^{\mi k \theta} + e^{-\mi k \theta} )
\end{aligned}
\end{equation*}
where we used the property $\displaystyle{\binom{\ell-3}{\ell-k-1} = \binom{\ell-3}{k-2}}$. Using $e^{\mi k\theta} + e^{-\mi k\theta} = 2\cos(k\theta)$, we finally obtain
\begin{equation*}
    \frac{\Re \bigl((1-e^{\mi \theta})^{\ell+1}\bigr)}{\sin^4 \left( \frac{\theta}{2} \right)} = 16 \sum_{k=2}^{\ell-1} \binom{\ell-3}{k-2} (-1)^k \cos (k \theta).
\end{equation*}
Then, we conclude
\begin{equation*}
\begin{aligned}
    \frac{1}{{\sin^4 \left( \frac{\theta}{2} \right)}} \sum_{\ell=3}^J 2^{-\ell} c_\ell^J \Re\bigl( (1-e^{\mi \theta})^{\ell+1} \bigr)  = \sum_{\ell=3}^J 2^{4-\ell} c_\ell^J \sum_{k=2}^{\ell-1} \binom{\ell-3}{k-2}(-1)^k \cos(k \theta) 
     = \sum_{k=2}^{J-1}  A_k^J \cos(k \theta) 
\end{aligned}
\end{equation*}
with $A_k^J$ as in \eqref{eq:25}.
\end{proof}
\noindent 
Formula \eqref{eq:25} can be inverted.
\begin{lemma}
For all integers $J \ge 2$ and  $3 \le \ell \le J$, we have
\begin{equation} \label{eq:26}
    c_\ell^J = (-1)^{\ell-1} 2^{\ell-4} \sum_{k=\ell-1}^{J-1} \binom{k-2}{\ell-3} A_k^J
\end{equation}
with $A_k^J$ as in \eqref{eq:25}
\end{lemma}
\begin{proof}
We use the binomial inversion formula \cite[Eq.~(5.48)]{GrahamKnuthPatashnik1994}, which we apply in the following form: for any finite sequences $\{x_n\}_n, \{y_m\}_m$ then it holds true
\begin{equation*}
    y_m = \sum_{n=m}^{J-3} \binom{n}{m} x_n \quad \text{if and only if} \quad x_n = \sum_{m=n}^{J-3} (-1)^{m-n} \binom{m}{n} y_m.
\end{equation*}
We apply this inversion formula to \eqref{eq:25} between the coefficients $\{A_k^J\}_{k=2}^{J-1}$ and $\{c_\ell^J\}_{\ell=3}^J$. Setting $x_n :=~c_{n+3}^J 2^{1-n}$, $y_m := (-1)^m A_{m+2}^J$, and inverting, we deduce
\begin{equation*}
    c_{n+3}^J 2^{1-n} = \sum_{m=n}^{J-3} (-1)^{m-n}\binom{m}{n} (-1)^m A_{m+2}^J = (-1)^n \sum_{m=n}^{J-3}\binom{m}{n} A_{m+2}^J.
\end{equation*}
Restoring the original indices via $\ell = n+3$ and $k = m+2$, we obtain \eqref{eq:26}.
\end{proof}
\begin{proof}[Proof of Theorem \ref{th:3.1}] From \eqref{eq:23}, we get
\begin{equation*}
    \Re(\delta_{\TTR_J}(e^{\mi \theta})) \ge 0 \quad \text{for all~} \theta \in [0,2\pi]
\end{equation*}
if and only if 
\begin{equation} \label{eq:27}
    T^J(\theta) \ge 0 \quad \text{for all} \quad  \theta \in [0,2\pi],
\end{equation}
with $T^J$ defined in \eqref{eq:24}.
We multiply \eqref{eq:27} by $1 + \cos \theta \ge 0$ and integrate in $[0,2\pi]$. Using the  relations
\begin{equation*}
    \frac{1}{2\pi}\int_{-\pi}^{\pi} \cos(k\theta) \, \dd \theta = 0 \qquad \frac{1}{2\pi}\int_{-\pi}^{\pi} \cos(k\theta) \cos \theta \, \dd \theta = \frac{1}{2}\delta_1^k \quad\quad \text{for all } k \ge 1,
\end{equation*}
we obtain
\begin{equation*}
    0 \le \frac{1}{2\pi}\int_{-\pi}^{\pi} T^J(\theta)(1+\cos\theta) \, \dd\theta = (4 - 2c_2^J) - 2c_2^J = 4 - 4c_2^J,
\end{equation*}
which yields $c_2^J \le 1$. Dividing in \eqref{eq:27} by $4-2c_2^J$, we get from Theorem \ref{th:4.1} that in order for \eqref{eq:27} to hold true, the inequality \eqref{eq:21} must be satisfied, or equivalently
\begin{equation*}
     \left|\frac{-2c_2^J}{2-c_2^J}\right| = \frac{2c_2^J}{2-c_2^J} \le 2\cos\left(\frac{\pi}{J+1}\right).
\end{equation*}
The function $f(t) = t/(2-t)$ is strictly increasing in $(0,1)$, therefore the maximum $c_2^J$ is reached when there is an equality to the upper bound. This gives
\begin{equation} \label{eq:28}
    c_2^J = 2 \frac{\cos x_J}{1 + \cos x_J}
\end{equation}
with $\displaystyle{x_J = \frac{\pi}{J+1}}$. The polynomial that attains equality is of the form \eqref{eq:22}. Normalizing the polynomial $T^J(\theta)$ by its constant term $4 - 2c_2^J$, the coefficient of the $\cos \theta$ term is then 
\begin{equation*}
    -\frac{2c_2^J}{2-c_2^J}.
\end{equation*}
This coefficient is strictly negative. On the other hand, the corresponding $k=1$ term in the optimal polynomial of Theorem \ref{th:4.1} is
\begin{equation*}
    2\cos\left(\frac{\pi}{J+1}\right) \cos(\theta - \psi).
\end{equation*}
To match the negative coefficient, the shifted cosine must absorb the negative sign, requiring $\cos(\theta - \psi) = -\cos \theta$. This gives $\psi = \pi$.  Consequently, this shift yields 
\begin{equation*}
    \cos(k(\theta - \pi)) = (-1)^k\cos(k\theta).
\end{equation*}
By the equality condition \eqref{eq:22} of Theorem \ref{th:4.1} applied to $n = J-1$, the coefficients must therefore satisfy
\begin{equation} \label{eq:29} 
    \frac{A_k^J}{4 - 2c_2^J} = \frac{2}{J+1} (-1)^k \left( (J - k)\cos(k x_J) + \frac{\sin((k+1)x_J)}{\sin x_J} \right), \quad \text{for } k=2,\ldots,J-1.
\end{equation}
Now we show that
\begin{equation} \label{eq:30}
    (J - k)\cos(k x_J) + \frac{\sin((k+1)x_J)}{\sin x_J} = \sin^2 x_J U'_{J-k}(\cos x_J).
\end{equation}
Indeed, using \eqref{eq:14}, we compute
\begin{equation*}
\begin{aligned}
    \sin^2 x_J U'_{J-k}(\cos x_J) & = \frac{\sin((J-k+1)x_J)\cos x_J - (J-k+1)\cos((J-k+1)x_J)\sin x_J}{\sin x_J}
    \\ & = \frac{\sin(\pi -k x_J)\cos x_J - (J-k+1)\cos(\pi - kx_J)\sin x_J}{\sin x_J}.
    \\ & = \frac{\sin(k x_J)\cos x_J + (J-k+1)\cos(k x_J)\sin x_J}{\sin x_J}
    \\ & = (J-k) \cos (k x_J) + \frac{\sin ((k+1)x_J)}{\sin x_J}.
\end{aligned}
\end{equation*}
Combining \eqref{eq:29} and \eqref{eq:30}, we deduce
\begin{equation*}
    A_k^J = (-1)^k \frac{2(4 - 2c_2^J)}{J+1} \sin^2 x_J U'_{J-k}(\cos x_J).
\end{equation*}
We obtain, substituting $c_2^J$ in \eqref{eq:28}, the expression
\begin{equation*}
    A_k^J = (-1)^k \frac{8(1-\cos x_J)}{J+1} U'_{J-k}(\cos x_J).
\end{equation*}
Finally, we plug this relation into \eqref{eq:26} and we get \eqref{eq:12}.
\end{proof}

\section{Numerical examples} \label{sec:5}

\noindent 
The numerical experiments were performed in MATLAB. The implementation has been done with the convolution quadrature library\footnote{\url{https://github.com/lehelb/TDBIE-CQ-book}} accompanying \cite{BanjaiSayas2022}.

\subsection{Example 1. Instabilities of the trapezoidal rule}

\noindent 
We consider a first test problem to illustrate the impact of a numerical perturbation on the stability of CQ methods.

\smallskip
\noindent
Let $\phi(t) = e^{-100\,(t - t_0)^2}$ be a Gaussian pulse centered at $t_0 = 0.5$ and consider the convolution equation $\K(\partial_t) g = \phi$ with kernel $\K(s) = s^{-1}$, whose exact solution is $g(t) = \dot \phi(t)$. To model a computational error $\varepsilon > 0$, we introduce the perturbed symbol $\K_\varepsilon(s) = \K(s) - \varepsilon$. The CQ scheme requires evaluating the inverse operator $\K_\varepsilon^{-1}(s) = \frac{s}{1 - \varepsilon s}$. This perturbed inverse has a simple real pole in the right half-plane at $s=1/\varepsilon$. Moreover, the kernels $\K$ and $\K_\varepsilon$ satisfy \eqref{eq:7} with $\mu'=0$ and $\K$ satisfies \eqref{eq:8bis} with $\zeta(s) = s/|s|$ and $\mu_2=1$.

\smallskip
\noindent 
To compute the convolution weights with formula \eqref{eq:2}, we need to evaluate the inverse symbol at the discrete Laplace variable $s^{\Delta t}(\zeta) = \delta(\zeta)/\Delta t$, where $\delta$ is the generating function of the underlying multistep method, and $\zeta$ ranges over a circle $|\zeta| = \lambda$ with $\lambda < 1$. A standard choice, used to balance round-off and truncation error in the inverse FFT, is $\lambda = \mathrm{eps}^{\,1/(2N+1)}$ (see, e.g., \cite[\S 3.4]{BanjaiSayas2022}) where $ N = T/\Delta t$ and $\mathrm{eps}$ is the machine precision. We compare the location of the set  $\{s^{\Delta t}(\zeta) : |\zeta| = \lambda\}$ with the pole $1/\varepsilon$ for each method.

\smallskip
\noindent 
\textbf{BDF2 and TTR$_\text{$J$}$.} Both methods are A-stable, and their generating functions $\delta$ satisfy~\eqref{eq:4}. In particular, we obtain
\begin{equation*}
    \sup_{|\zeta| \le 1} |\delta_{\mathrm{BDF2}}(\zeta)| = \delta_{\mathrm{BDF2}}(-1) = 4, \qquad \sup_{|\zeta| \le 1} |\delta_{\TTR_J}(\zeta)| = \delta_{\TTR_J}(-1) \le \frac{160(J+1)}{9\pi^2}
\end{equation*}
where we used Lemma \ref{lem:3.8} and Proposition \ref{prop:3.6} (see also Remark \ref{rem:3.7}).\footnote{We remark that Lemma \ref{lem:3.8} implies $\sup_{|\zeta|\le 1} |\delta_{\TTR_J}(\zeta)| = \delta_{\TTR_J}(-1)$ for optimal coefficients $\{c_\ell^J\}_{\ell=2}^J$ only for $J=2,\ldots,6$ (see Table \ref{tab:2}), but we conjecture it is true for all $J \ge 2$.} Consequently, we have
\begin{equation*}
    \sup_{|\zeta| = \lambda} \bigl|s_{\mathrm{BDF2}}^{\Delta t}(\zeta)\bigr| \le \frac{4}{\Delta t}, \quad\quad \sup_{|\zeta| = \lambda} \bigl|s_{\TTR_J}^{\Delta t}(\zeta)\bigr| \le \frac{160(J+1)}{9\pi^2\Delta t}.
\end{equation*}
\textbf{Trapezoidal rule.} A direct computation shows that
\begin{equation*}
    \sup_{|\zeta|=\lambda}|s_{\TR}^{\Delta t}(\zeta)| 
    = s_{\TR}^{\Delta t}(-\lambda) 
    = \frac{2(1+\lambda)}{\Delta t\,(1-\lambda)}.
\end{equation*}
Using $1 - \lambda \le |\log\mathrm{eps}|/(2N) = \Delta t|\log\mathrm{eps}|/(2T)$, we obtain
\begin{equation*}
    \sup_{|\zeta|=\lambda}|s_{\TR}^{\Delta t}(\zeta)| 
    \ge \frac{4T}{\Delta t^{2}|\log\mathrm{eps}|},
\end{equation*}
which grows as $\Delta t^{-2}$, in contrast with the $\Delta t^{-1}$ bounds obtained for BDF2 and TTR$_\text{J}$.

\smallskip
\noindent 
For $T=10$, $N = 10^{3}$ and machine precision $\mathrm{eps} \approx 2.22 \cdot 10^{-16}$, the three methods give
\begin{equation*}
    \sup_{|\zeta|=\lambda}|s_{\mathrm{BDF2}}^{\Delta t}(\zeta)| \le 4\cdot 10^2, \quad  \sup_{|\zeta|=\lambda}|s_{\TTR_J}^{\Delta t}(\zeta)| \le 2(J+1) \cdot 10^2, \quad \sup_{|\zeta|=\lambda}|s_{\TR}^{\Delta t}(\zeta)| \ge 1.11\cdot 10^{4}.
\end{equation*}
The location of these thresholds is shown in Figure \ref{fig:thresholds}.

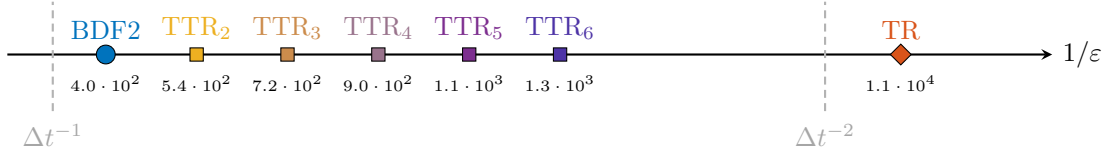
\begin{figure}[H]
\centering
\begin{tikzpicture}[>=stealth, font=\small]
\definecolor{colBDF2}{RGB}{0,115,189}
\definecolor{colTR}{RGB}{217,84,26}
\definecolor{colTTRa}{RGB}{237,176,33}
\definecolor{colTTRb}{RGB}{204,140,76}
\definecolor{colTTRc}{RGB}{153,115,140}
\definecolor{colTTRd}{RGB}{125,46,143}
\definecolor{colTTRe}{RGB}{77,51,166}
\draw[->, thick] (-0.3,0) -- (13.5,0) node[right] {$1/\varepsilon$};
\def\xB{1.0}
\def\xT{2.2}
\def\xTb{3.4}
\def\xTc{4.6}
\def\xTd{5.8}
\def\xTe{7.0}
\def\xTR{11.5}
\def\xDt{0.3}
\def\xDtSq{10.5}
\draw[densely dashed, gray!60, thick] (\xDt,-0.75) -- (\xDt,0.7);
\draw[densely dashed, gray!60, thick] (\xDtSq,-0.75) -- (\xDtSq,0.7);
\node[gray!70, below=1pt] at (\xDt,-0.75)   {$\Delta t^{-1}$};
\node[gray!70, below=1pt] at (\xDtSq,-0.75) {$\Delta t^{-2}$};
\node[circle, fill=colBDF2,  draw=black, inner sep=2.5pt] at (\xB,0)  {};
\node[rectangle, fill=colTTRa, draw=black, inner sep=2.5pt] at (\xT,0)  {};
\node[rectangle, fill=colTTRb, draw=black, inner sep=2.5pt] at (\xTb,0) {};
\node[rectangle, fill=colTTRc, draw=black, inner sep=2.5pt] at (\xTc,0) {};
\node[rectangle, fill=colTTRd, draw=black, inner sep=2.5pt] at (\xTd,0) {};
\node[rectangle, fill=colTTRe, draw=black, inner sep=2.5pt] at (\xTe,0) {};
\node[diamond,   fill=colTR,   draw=black, inner sep=2pt]   at (\xTR,0) {};
\node[above=2pt, colBDF2] at (\xB,0)  {BDF2};
\node[above=2pt, colTTRa] at (\xT,0)  {TTR$_2$};
\node[above=2pt, colTTRb] at (\xTb,0) {TTR$_3$};
\node[above=2pt, colTTRc] at (\xTc,0) {TTR$_4$};
\node[above=2pt, colTTRd] at (\xTd,0) {TTR$_5$};
\node[above=2pt, colTTRe] at (\xTe,0) {TTR$_6$};
\node[above=2pt, colTR]   at (\xTR,0) {TR};
\node[below=4pt] at (\xB,0)  {\tiny$4.0 \cdot 10^2$};
\node[below=4pt] at (\xT,0)  {\tiny $5.4 \cdot 10^2$};
\node[below=4pt] at (\xTb,0) {\tiny $7.2\cdot 10^2$};
\node[below=4pt] at (\xTc,0) {\tiny $9.0 \cdot 10^2$};
\node[below=4pt] at (\xTd,0) {\tiny $1.1 \cdot 10^3$};
\node[below=4pt] at (\xTe,0) {\tiny $1.3 \cdot 10^3$};
\node[below=4pt] at (\xTR,0) {\tiny $1.1 \cdot 10^4$};
\end{tikzpicture}
\caption{Thresholds for $\sup_{|\zeta|=\lambda}|s^{\Delta t}(\zeta)|$ for $T = 10$, $N = 10^3$, $\mathrm{eps} \approx 2.22 \cdot 10^{-16}$.}
\label{fig:thresholds}
\end{figure}

\noindent 
The BDF2 and TTR$_J$ thresholds are of order $\Delta t^{-1}$, with TTR$_J$ growing linearly in $J$. The trapezoidal rule threshold is of order $\Delta t^{-2}$. A method is unstable when $1/\varepsilon$ lies to the left of its threshold, since the discrete Laplace variable then samples $\K_\varepsilon^{-1}$ arbitrarily close to its pole.

\smallskip
\noindent
We report in Figure~\ref{fig:perturb} the pointwise error
\begin{equation*}
    \bigl|\K_\varepsilon^{-1}(\partial_t^{\Delta t})\phi\,(t_n) - \K^{-1}(\partial_t)\phi\,(t_n)\bigr|, \qquad n = 0,\ldots,N,
\end{equation*}
for BDF2, TR, and TTR$_4$ (with optimal coefficients given in Table \ref{tab:2})  with $T=10$ and $N=10^3$. 

\begin{figure}[H]
\centering
\begin{subfigure}{0.48\linewidth}
    \centering
    \includegraphics[width=\linewidth]{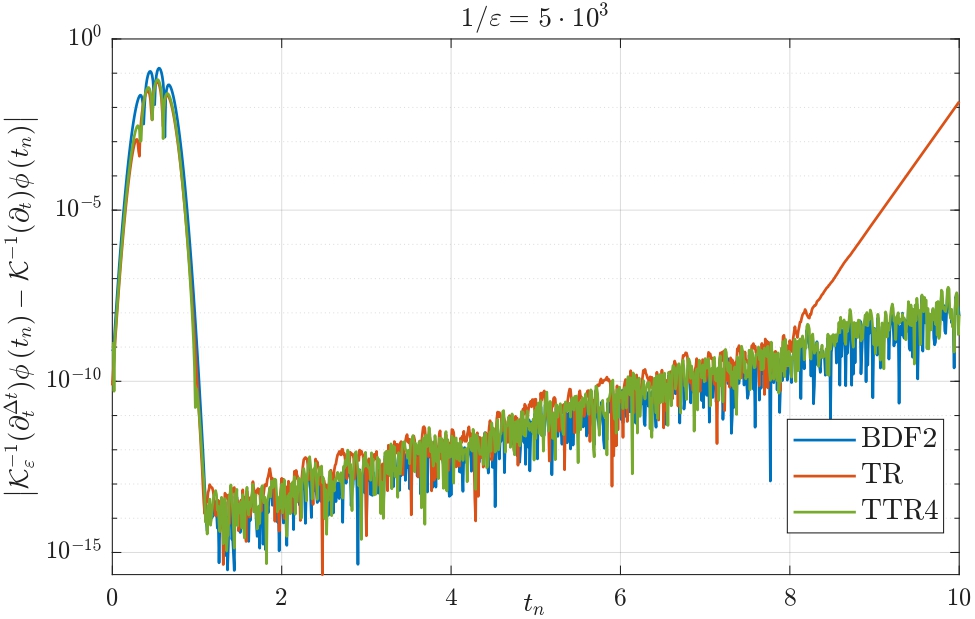}
\end{subfigure}
\hfill
\begin{subfigure}{0.48\linewidth}
    \centering
    \includegraphics[width=\linewidth]{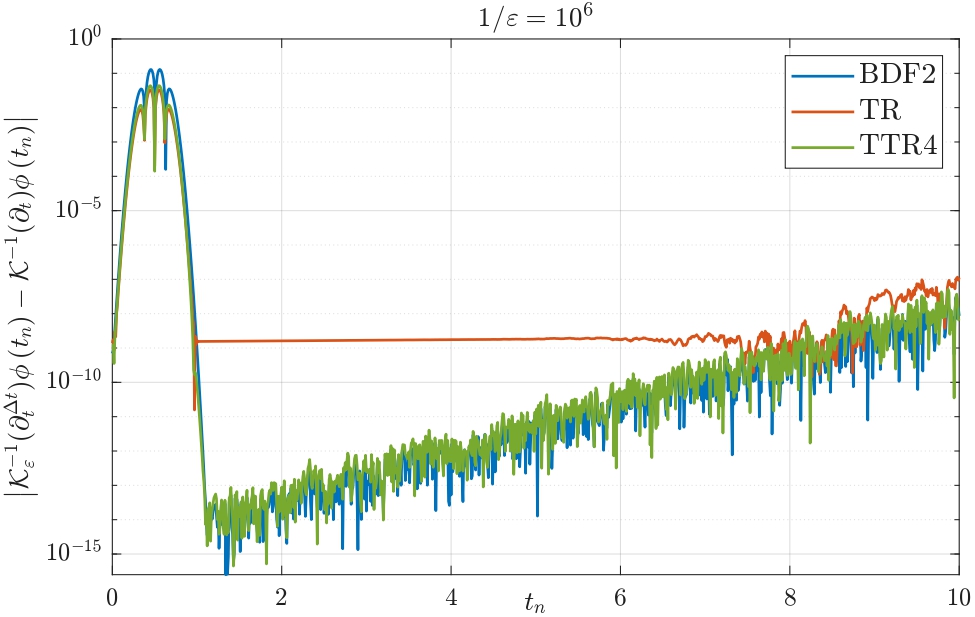}
\end{subfigure}
\caption{Example 1. Pointwise error
$\bigl|\K_\varepsilon^{-1}(\partial_t^{\Delta t})\phi\,(t_n) - \K^{-1}(\partial_t)\phi\,(t_n)\bigr|$ for BDF2, TR, and TTR$_4$, with $T=10$ and $N=10^3$. On the left, $1/\varepsilon = 5 \cdot 10^3$ and on the right, $1/\varepsilon = 10^6$.}
\label{fig:perturb}
\end{figure}

\noindent
In Figure~\ref{fig:perturb} (left), $1/\varepsilon = 5 \cdot 10^3$ lies between the BDF2 and TTR$_4$ thresholds and the TR threshold. TR is unstable, while BDF2 and TTR$_4$ remain accurate. In Figure~\ref{fig:perturb} (right), $1/\varepsilon = 10^6$ lies beyond all thresholds: all the methods perform accurately.

\subsection{Example 2. Convergence rates}

\noindent
We test the predicted convergence rate and the explicit dependence on $J$ of the principal error constant \eqref{eq:15}. We consider the hyperbolic symbol $\K(s) = s^{1.5}$, and the Gaussian datum $g(t) = e^{-100(t-0.5)^2}$ on $[0,T]$ with $T=1$. The datum $g$ is smooth and zero, up to machine precision, at $t = 0$ and $t = T$, so \eqref{eq:5}--\eqref{eq:6} apply for every $r \in \N$ and CQ based on BDF2, TR, and TTR$_J$ all attain the second-order rate. We compute the discrete convolution $\K(\partial_t^{\Delta t}) g (T)$ for $N = 2^k$ with $k=5,6,7,8,9$ time steps using BDF2, the trapezoidal rule, and the truncated trapezoidal rule of order $J = 2, 3, 4$, with the optimal coefficients given by Theorem~\ref{th:3.1} (see Table~\ref{tab:2}). We use as a reference value a CQ approximation computed with TTR$_8$ at $N_{\mathrm{ref}} = 2^{14}$. Figure~\ref{fig:2} reports the absolute error $\bigl|\K(\partial_t) g (T) - \K(\partial_t^{\Delta t}) g (T)\bigr|$
as a function of $\Delta t$ for the five methods. All curves exhibit the expected $\mathcal{O}(\Delta t^2)$ behavior, and the relative ordering matches the principal error constants in \eqref{eq:15}: TR attains the smallest error, followed by TTR$_4$, TTR$_3$, TTR$_2$, and finally BDF2, in agreement with the theory.

\begin{figure}[H]
    \centering
    \includegraphics[width=0.7\linewidth]{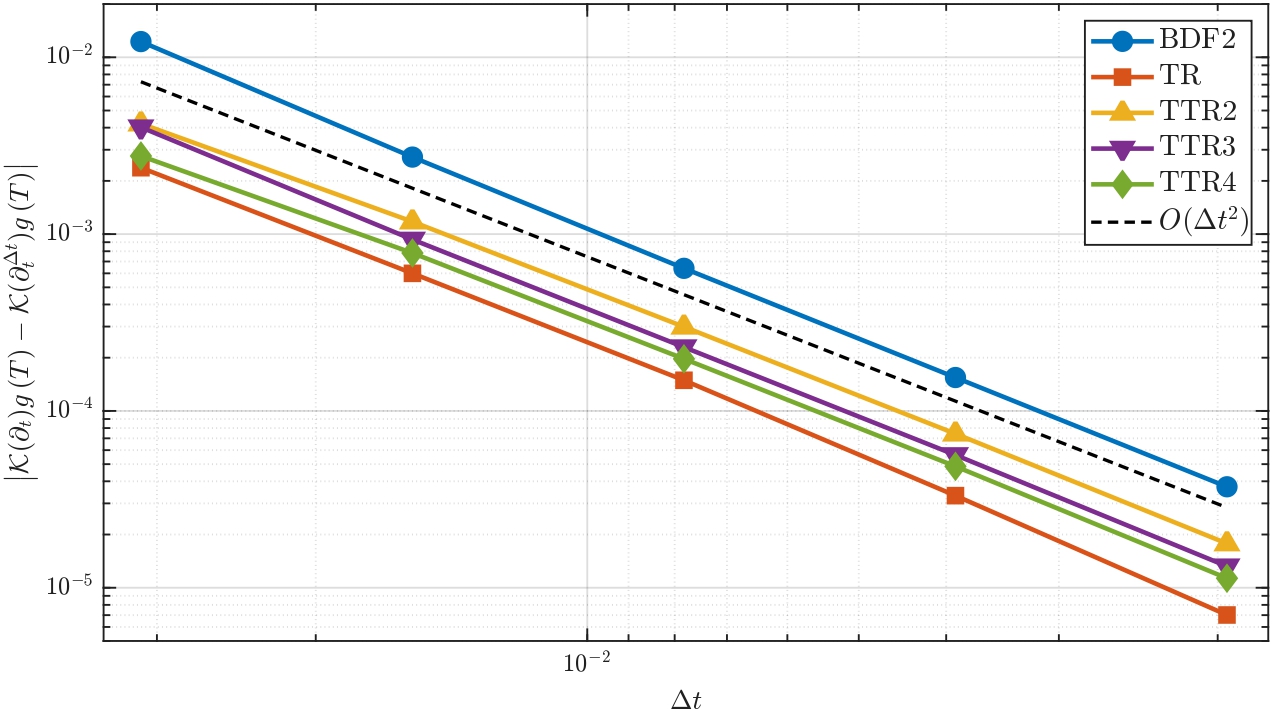}
    \caption{Example 2. Convergence of the discrete convolution
    $\K(\partial_t^{\Delta t}) g (T)$ to $\K(\partial_t) g (T)$ for
    $\K(s) = s^{1.5}$ and $g(t) = e^{-100(t-0.5)^2}$, with $T=1$.
    All methods exhibit the expected $\mathcal{O}(\Delta t^2)$ rate;
    the relative ordering of the curves reflects the principal error
    constants \eqref{eq:15}.}
    \label{fig:2}
\end{figure}
 
\section{Conclusion}

\noindent
We have studied the truncated trapezoidal rule, the family of A-stable linear multistep methods introduced in \cite{Banjai2022} as a compromise between BDF2 and the trapezoidal rule. For every integer $J \ge 2$, we derived an explicit closed-form expression for the coefficients $\{c_\ell^J\}_{\ell=2}^J$ that minimize the principal error constant under the A-stability constraint (Theorem~\ref{th:3.1}). As a direct consequence, the optimal principal error constant admits the closed form \eqref{eq:15} which is monotonically decreasing in $J$ and converges to the optimal Dahlquist value $1/12$ as $J \to \infty$. For every finite $J \ge 2$, the constant is strictly smaller than the BDF2 value $1/3$. We have further shown that the leading dissipation coefficient vanishes when $J$ is even. Within CQ, this makes the TTR$_J$ an alternative to both BDF2 and the trapezoidal rule. The analyticity of $\delta_{\TTR_J}$ in a neighbourhood of $\{|\zeta| \le 1\}$ ensures the milder regularity and perturbation requirements summarized in Table~\ref{tab:1}, while the explicit control on the error constant allows one to approach the optimal Dahlquist value $1/12$ as closely as desired by increasing $J$. The numerical experiments of Section~\ref{sec:5} confirm this. TTR$_J$-based CQ is stable under symbol perturbations, while the trapezoidal rule is not: its pole at $\zeta = -1$ amplifies the perturbations. The error constant of TTR$_J$ is also smaller than that of BDF2. Future work includes the development of fast algorithms for this method, along the lines of fast and oblivious CQ \cite{SchadleLopezFernandezLubich2006} or parsimonious CQ \cite{MelenkNick2024}. A further open problem is the analysis of a generalised CQ version of the TTR$_J$ on non-uniform meshes \cite{LopezFernandezSauter2013,LopezFernandezSauter2015a,BanjaiFerrari2026}.

\section*{Acknowledgments}

\noindent
Most of this work has been carried out while the author was affiliated with the Faculty of Mathematics at the University of Vienna. The author is member of the Gruppo Nazionale Calcolo Scientifico-Istituto Nazionale di Alta Matematica (GNCS-INdAM). The author thanks Lehel Banjai for useful discussions.

\appendix
\section{Auxiliary results}

\noindent 
In this appendix, we collect some identities involving the Chebyshev polynomials of the second kind $\{U_k\}_{k \ge 0}$ that are used in the proofs of Section~\ref{sec:3}. In particular, Lemmas~\ref{lem:A.4} and~\ref{lem:A.5} are used, respectively, in the proof of Proposition~\ref{prop:3.6} and in the derivation of formula~\eqref{eq:16}.

\smallskip
\noindent
We recall the recurrence relation (see, e.g., \cite[Eq. 1.6a]{Mason2002})
\begin{equation} \label{eq:31}
    U_{k+1}(y) = 2y U_k(y) - U_{k-1}(y)
\end{equation}
which together with the initial conditions $U_0(x) = 1$ and $U_1(x) = 2x$ provide an alternative definition of the polynomials $\{U_k\}_{k \ge 0}$.
\begin{lemma}
For every integer $J \ge 2$, it holds true
\begin{equation} \label{eq:32}
    \sum_{k=0}^{J-2} U_k(y) = \frac{1 - U_{J-1}(y) + U_{J-2}(y)}{2(1-y)}, \quad y \in [-1,1).
\end{equation}
\end{lemma}
\begin{proof}
From the three-term recurrence \eqref{eq:31}, we deduce
\begin{equation*}
    2(1-y) U_k(y) = 2 U_k(y) - \bigl( U_{k+1}(y) + U_{k-1}(y) \bigr) = \bigl( U_k(y) - U_{k+1}(y) \bigr) - \bigl( U_{k-1}(y) - U_k(y) \bigr).
\end{equation*}
Summing for $k=0,1,\ldots,J-2$, the right-hand side telescopes
\begin{equation*}
    2(1-y) \sum_{k=0}^{J-2} U_k(y) = \bigl( U_0(y) - U_{J-1}(y) \bigr) - \bigl( U_{-1}(y) - U_{J-2}(y) \bigr).
\end{equation*}
With $U_0 \equiv 1$ and $U_{-1} \equiv 0$, dividing by $2(1-y)$ yields \eqref{eq:32}.
\end{proof}
\begin{lemma}
For all $J \ge 2$ integer, it holds true
\begin{equation} \label{eq:33}
    \sum_{k=0}^{J-2} (-1)^k U_k(y) = \frac{1 + (-1)^J \bigl( U_{J-1}(y) + U_{J-2}(y) \bigr)}{2(1+y)}, \quad\quad y \in (-1,1].
\end{equation}
\end{lemma}
\begin{proof}
From the three-term recurrence \eqref{eq:31} we deduce
\begin{equation*}
    2(1+y)(-1)^k U_k(y) = (-1)^k \bigl( U_{k+1}(y) + U_{k-1}(y) \bigr) + 2(-1)^k U_k(y),
\end{equation*}
so that, rearranging,
\begin{equation*}
    2(1+y)(-1)^k U_k(y) = \bigl( (-1)^k U_k(y) - (-1)^{k+1} U_{k+1}(y) \bigr) + \bigl((-1)^k U_k(y) - (-1)^{k-1} U_{k-1}(y) \bigr).
\end{equation*}
Summing for $k=0,1,\ldots,J-2$, the right-hand side telescopes
\begin{equation*}
    2(1+y) \sum_{k=0}^{J-2} (-1)^k U_k(y) = \bigl( U_0(y) - (-1)^{J-1} U_{J-1}(y) \bigr) + \bigl( (-1)^{J-2} U_{J-2}(y) - (-1)^{-1} U_{-1}(y) \bigr).
\end{equation*}
With $U_0 \equiv 1$, $U_{-1} \equiv 0$, dividing by $2(1+y)$ yields \eqref{eq:33}.
\end{proof}
\begin{lemma}
Let $J \ge 2$ be an integer and $x_J = \displaystyle{\frac{\pi}{J+1}}$. Then, the following identities are true
\begin{align} \label{eq:34}
    U_J(\cos x_J) & = 0,
    \\ U_{J-1}(\cos x_J) &  = 1, \label{eq:35}
    \\ U_{J-2}(\cos x_J) & = 2 \cos x_J,  \label{eq:36}
    \\ U_{J-3}(\cos x_J) & = 4\cos^2 x_J - 1, \label{eq:37}
    \\ U'_{J-1}(\cos x_J) & = \frac{(J+1)\cos x_J}{1-\cos^2 x_J}, \label{eq:38}
    \\ U'_{J-2}(\cos x_J) & = \frac{2J\cos^2 x_J - J + 1}{1-\cos^2 x_J}, \label{eq:39}
\end{align}
with the convention $U_{-1}(y) = 0$.
\end{lemma}
\begin{proof}
Recalling definition \eqref{eq:13}, we readily obtain
\begin{equation*}
\begin{aligned}
    U_J(\cos x_J) & = \frac{\sin((J+1) x_J)}{\sin x_J} = \frac{\sin(\pi)}{\sin x_J} = 0,
    \\ U_{J-1}(\cos x_J) & = \frac{\sin(J x_J)}{\sin x_J} = \frac{\sin(\pi-x_J)}{\sin x_J} = 1,
    \\ U_{J-2}(\cos x_J) & = \frac{\sin((J-1)x_J)}{\sin x_J} = \frac{\sin(\pi-2x_J)}{\sin x_J} = 2\cos x_J,
    \\ U_{J-3}(\cos x_J) & = \frac{\sin((J-2)x_J)}{\sin x_J} = \frac{\sin(\pi-3x_J)}{\sin x_J} = 4\cos^2 x_J - 1.
\end{aligned}
\end{equation*}
Exploiting the recurrence relation \cite[Ex. 15 Ch. 2]{Mason2002}
\begin{equation*}
    U_k'(y) = \frac{(k+2) U_{k-1}(y) - k\, U_{k+1}(y)}{2(1-y^2)},
\end{equation*}
we compute with \eqref{eq:36} and \eqref{eq:34}
\begin{equation*}
    U'_{J-1}(\cos x_J) = \frac{(J+1) U_{J-2}(\cos x_J) - (J-1) U_J(\cos x_J)}{2(1-\cos^2 x_J)} = \frac{(J+1)\cos x_J}{1-\cos^2 x_J}.
\end{equation*}
Similarly, we compute with \eqref{eq:37} and \eqref{eq:35}
\begin{equation*}
    U'_{J-2}(\cos x_J) = \frac{J\, U_{J-3}(\cos x_J) - (J-2) U_{J-1}(\cos x_J)}{2(1-\cos^2 x_J)} = \frac{2J\cos^2 x_J - J + 1}{1-\cos^2 x_J}.
\end{equation*}
\end{proof}
\begin{lemma} \label{lem:A.4}
Let $J \ge 2$ be an integer and $x_J = \displaystyle\frac{\pi}{J+1}$. Then, it holds true
\begin{equation*}
    \sum_{k=1}^{J-2} U'_k(\cos x_J) = \frac{2(J+1)\cos^2 x_J - (J-1)(1+\cos x_J)}{2(1-\cos x_J)^2(1+\cos x_J)}.
\end{equation*}
\end{lemma}
\begin{proof}
Differentiating \eqref{eq:32} and using $U'_0(y) = 0$,
\begin{equation*}
    \sum_{k=1}^{J-2} U'_k(y) = \frac{2(1-y)\bigl(U'_{J-2}(y) - U'_{J-1}(y)\bigr) + 2\bigl(1 - U_{J-1}(y) + U_{J-2}(y)\bigr)}{4(1-y)^2}.
\end{equation*}
At $y = \cos x_J$, identities \eqref{eq:35}, \eqref{eq:36}, \eqref{eq:38} and \eqref{eq:39} give
\begin{equation*}
\begin{aligned}
    \sum_{k=1}^{J-2} U'_k(\cos x_J) & = \frac{2\bigl(2J\cos^2 x_J - (J+1)\cos x_J - (J-1)\bigr) + 4\cos x_J (1+\cos x_J)}{4(1-\cos x_J)^2(1+\cos x_J)}
\end{aligned}
\end{equation*}
which gives the claim.
\end{proof}
\begin{lemma} \label{lem:A.5}
Let $J \ge 2$ be an integer and $x_J = \displaystyle\frac{\pi}{J+1}$. Then, it holds true
\begin{equation*}
    \sum_{k=1}^{J-2} (-1)^k U'_k(\cos x_J) = \frac{(-1)^J (J+1)(2\cos x_J - 1)}{2(1-\cos x_J)(1+\cos x_J)} + \frac{(-1)^J - 1}{2(1+\cos x_J)^2}.
\end{equation*}
\end{lemma}
\begin{proof}
Differentiating \eqref{eq:33} and using $U'_0(y) = 0$, we obtain
\begin{equation*}
    \sum_{k=1}^{J-2}(-1)^k U'_k(y) = \frac{(-1)^J\bigl(U'_{J-1}(y) + U'_{J-2}(y)\bigr) 2(1+y) - 2\bigl(1 + (-1)^J(U_{J-1}(y) + U_{J-2}(y))\bigr)}{4(1+y)^2}.
\end{equation*}
At $y = \cos x_J$, identities \eqref{eq:35}, \eqref{eq:36}, \eqref{eq:38} and \eqref{eq:39} yield
\begin{equation*}
\begin{aligned}
    \sum_{k=1}^{J-2}(-1)^k U'_k(\cos x_J) & = \frac{(-1)^J\bigl(J(\cos x_J + 2\cos^2 x_J - 1) + 2\cos^2 x_J\bigr) - (1-\cos x_J)}{2(1-\cos x_J)(1+\cos x_J)^2}
    \\ & = \frac{(-1)^J\bigl((J+1)(2\cos x_J-1)(1+\cos x_J) + 1 - \cos x_J \bigr) - (1-\cos x_J)}{2(1-\cos x_J)(1+\cos x_J)^2},
\end{aligned}
\end{equation*}
which gives the claim.
\end{proof}

\bibliographystyle{plain} 
\bibliography{bibliography}

@article{MelenkNick2024,
      title={Parsimonious convolution quadrature}, 
      author={Melenk, J. M. and Nick, J.},
      year={2024},
      journal = {Preprint. arXiv: 2410.15079},
      url={https://arxiv.org/abs/2410.15079}, 
}

@article{SchadleLopezFernandezLubich2006,
  author  = {Sch\"adle, A. and L\'opez-Fern\'andez, M. and Lubich, C.},
  title   = {Fast and Oblivious Convolution Quadrature},
  journal = {SIAM J. Sci. Comput.},
  volume  = {28},
  number  = {2},
  pages   = {421--438},
  year    = {2006},
  doi     = {10.1137/050623139}
}

@article{BanjaiFerrari2026,
  author  = {Banjai, L. and Ferrari, M.},
  title   = {Generalized convolution quadrature based on the trapezoidal rule},
  journal = {IMA J. Numer. Anal.},
  year    = {2026},
}

@article {ErusluSayas2020,
    AUTHOR = {Eruslu, H. and Sayas, F.-J.},
     TITLE = {Polynomially bounded error estimates for trapezoidal rule
              convolution quadrature},
   JOURNAL = {Comput. Math. Appl.},
    VOLUME = {79},
      YEAR = {2020},
    NUMBER = {6},
     PAGES = {1634--1643},
}

@article {BanjaiLubich2011,
    AUTHOR = {Banjai, L. and Lubich, C.},
     TITLE = {An error analysis of {R}unge-{K}utta convolution quadrature},
   JOURNAL = {BIT Numer. Math.},
    VOLUME = {51},
      YEAR = {2011},
    NUMBER = {3},
     PAGES = {483--496},
}

@article {LubichOstermann1993,
    AUTHOR = {Lubich, C. and Ostermann, A.},
     TITLE = {Runge-{K}utta methods for parabolic equations and convolution
              quadrature},
   JOURNAL = {Math. Comp.},
    VOLUME = {60},
      YEAR = {1993},
    NUMBER = {201},
     PAGES = {105--131},
}

@article {Lubich2004,
    AUTHOR = {Lubich, C.},
     TITLE = {Convolution quadrature revisited},
   JOURNAL = {BIT Numer. Math.},
    VOLUME = {44},
      YEAR = {2004},
    NUMBER = {3},
     PAGES = {503--514},
}

@article {Banjai2010,
    AUTHOR = {Banjai, L.},
     TITLE = {Multistep and multistage convolution quadrature for the wave
              equation: algorithms and experiments},
   JOURNAL = {SIAM J. Sci. Comput.},
    VOLUME = {32},
      YEAR = {2010},
    NUMBER = {5},
     PAGES = {2964--2994},
}

@article{BanjaiFerrari2022,
    AUTHOR  = {Banjai, L. and Ferrari, M.},
    TITLE   = {Runge--{K}utta convolution quadrature based on {G}auss methods},
    JOURNAL = {Numer. Math.},
    FJOURNAL = {Numerische Mathematik},
    VOLUME  = {156},
    NUMBER  = {5},
    PAGES   = {1719--1750},
    YEAR    = {2024},
    DOI     = {10.1007/s00211-024-01429-4},
}

@article {Lubich1994,
    AUTHOR = {Lubich, C.},
     TITLE = {On the multistep time discretization of linear
              initial-boundary value problems and their boundary integral
              equations},
   JOURNAL = {Numer. Math.},
    VOLUME = {67},
      YEAR = {1994},
    NUMBER = {3},
     PAGES = {365--389},
}

@article {BanjaiLubichMelenk2011,
    AUTHOR = {Banjai, L. and Lubich, C. and Melenk, J. M.},
     TITLE = {Runge-{K}utta convolution quadrature for operators arising in
              wave propagation},
   JOURNAL = {Numer. Math.},
    VOLUME = {119},
      YEAR = {2011},
    NUMBER = {1},
     PAGES = {1--20},
}

@article {Lubich1988b,
    AUTHOR = {Lubich, C.},
     TITLE = {Convolution quadrature and discretized operational calculus.
              {II}},
   JOURNAL = {Numer. Math.},
    VOLUME = {52},
      YEAR = {1988},
    NUMBER = {4},
     PAGES = {413--425},
}

@article {Lubich1988a,
    AUTHOR = {Lubich, C.},
     TITLE = {Convolution quadrature and discretized operational calculus.
              {I}},
   JOURNAL = {Numer. Math.},
    VOLUME = {52},
      YEAR = {1988},
    NUMBER = {2},
     PAGES = {129--145},
}

@article {LopezFernandezSauter2013,
    AUTHOR = {Lopez-Fernandez, M. and Sauter, S.},
     TITLE = {Generalized convolution quadrature with variable time
              stepping},
   JOURNAL = {IMA J. Numer. Anal.},
    VOLUME = {33},
      YEAR = {2013},
    NUMBER = {4},
     PAGES = {1156--1175},
}

@book {BanjaiSayas2022,
    AUTHOR = {Banjai, L. and Sayas, F.-J.},
     TITLE = {Integral Equation Methods for Evolutionary PDE: A Convolution Quadrature Approach},
    SERIES = {Springer Series in Computational Mathematics},
    VOLUME = {59},
 PUBLISHER = {Springer},
   ADDRESS = {Cham},
      YEAR = {2022},
}

@article {LopezFernandezSauter2015a,
    AUTHOR = {Lopez-Fernandez, M. and Sauter, S.},
     TITLE = {Generalized convolution quadrature with variable time
              stepping. {P}art {II}: {A}lgorithm and numerical results},
   JOURNAL = {Appl. Numer. Math.},
    VOLUME = {94},
      YEAR = {2015},
     PAGES = {88--105},
}

@book{GrahamKnuthPatashnik1994,
    AUTHOR = {Graham, R. L. and Knuth, D. E. and Patashnik, O.},
     TITLE = {Concrete Mathematics: A Foundation for Computer Science},
   EDITION = {2},
 PUBLISHER = {Addison-Wesley},
   ADDRESS = {Reading, MA},
      YEAR = {1994},
}

@book{Rudin1987,
    AUTHOR = {Rudin, W.},
     TITLE = {Real and Complex Analysis},
   EDITION = {3},
 PUBLISHER = {McGraw-Hill},
   ADDRESS = {New York},
      YEAR = {1987},
}

@article{egervary-szasz-1928,
  author  = {Egerv{\'a}ry, E. v. and Sz{\'a}sz, O.},
  title   = {Einige Extremalprobleme im Bereiche der trigonometrischen
             Polynome},
  journal = {Math. Z.},
  volume  = {27},
  year    = {1928},
  pages   = {641--652}
}

@article{fejer1915uber,
    AUTHOR = {Fej{\'e}r, L.},
     TITLE = {{\"U}ber trigonometrische {P}olynome},
   JOURNAL = {J. Reine Angew. Math.},
    VOLUME = {146},
     PAGES = {53--82},
      YEAR = {1915},
}

@article{dahlquist1963special,
    AUTHOR = {Dahlquist, G.},
     TITLE = {A special stability problem for linear multistep methods},
   JOURNAL = {BIT Numer. Math.},
    VOLUME = {3},
    NUMBER = {1},
     PAGES = {27--43},
      YEAR = {1963},
}

@book{Mason2002,
    AUTHOR = {Mason, J. C. and Handscomb, D. C.},
     TITLE = {Chebyshev Polynomials},
 PUBLISHER = {CRC Press},
   ADDRESS = {Boca Raton},
      YEAR = {2002},
}

@article {Banjai2022,
    AUTHOR = {Banjai, L.},
     TITLE = {Implicit/explicit, {BEM}/{FEM} coupled scheme for acoustic
              waves with the wave equation in the second order formulation},
   JOURNAL = {Comput. Methods Appl. Math.},
    VOLUME = {22},
      YEAR = {2022},
    NUMBER = {4},
     PAGES = {757--773},
}

\end{document}